\documentclass[a4,reqno,11pt]{article}
\usepackage{graphics, color, amsmath, amsfonts, amssymb, amsthm, amscd}
\usepackage{epsf}
\usepackage{psfrag}

\usepackage{appendix}
\usepackage{graphicx}
\usepackage[colorlinks]{hyperref}
\DeclareGraphicsExtensions{.ps,.eps}

\usepackage[latin1]{inputenc}

\usepackage{natbib}

\usepackage{amsmath}
\usepackage{amsfonts}
\usepackage{amssymb}
\usepackage{amsthm}

\numberwithin{equation}{section}

\newtheorem{theorem}{Theorem}[section]

\newtheorem{lemma}{Lemma}[section]
\newtheorem{proposition}{Proposition}[section]

\newtheorem{remark}{Remark}

\newcommand{\calD}{\mathcal{D}}

\newcommand{\calH}{\mathcal{H}}

\newcommand{\calU}{\mathcal{U}}

\newcommand{\frh}{\mathfrak{h}}

\newcommand{\frt}{\mathfrak{t}}

\newcommand{\bbD}{\mathbb{D}}
\newcommand{\bbE}{\mathbb{E}}

\newcommand{\bbH}{\mathbb{H}}

\newcommand{\bbN}{\mathbb{N}}

\newcommand{\bbP}{\mathbb{P}}
\newcommand{\bbQ}{\mathbb{Q}}
\newcommand{\bbR}{\mathbb{R}}
\newcommand{\bbS}{\mathbb{S}}

\newcommand{\bbZ}{\mathbb{Z}}
\newcommand{\sfz}{{\sf z}}

\newcommand{\sfe}{{\sf e}}
\newcommand{\Zd}{\bbZ^d}
\newcommand{\Rd}{\bbR^d}

\newcommand{\B}{\mathbf{B}}

\newcommand{\sur}{\mathfrak{s}_h}

\newcommand{\fcone}{Y^>_\delta(h)}
\newcommand{\fconeh}[1]{Y^>_\delta(#1)}
\newcommand{\tfcone}{Y^>_{3\delta}(h)}
\newcommand{\tcone}{Y_{3\delta}(h)}
\newcommand{\bcone}{Y^<_\delta(h)}
\newcommand{\tbcone}{Y^<_{3\delta}(h)}
\newcommand{\dfcone}{Y^>_{2\delta}(h)}
\newcommand{\dbcone}{Y^<_{2\delta}(h)}

\newcommand{\lb}{\left(}
\newcommand{\rb}{\right)}
\newcommand{\lbr}{\left\{}
\newcommand{\rbr}{\right\}}

\newcommand{\step}[1]{S{\small TEP}\,#1}

\def\1{\ifmmode {1\hskip -3pt \rm{I}}
\else {\hbox {$1\hskip -3pt \rm{I}$}}\fi} %indicator

\newcommand{\smallo}{{\mathrm o}}
\newcommand{\so}{\smallo (1)}

\newcommand{\df}{\stackrel{\Delta}{=}}

\newcommand{\Zn}{Z_n}
\newcommand{\Znf}[1]{Z_n^{#1}}

\newcommand{\Plf}[1]{\bbP_\lambda^{#1}}
\newcommand{\Pnf}[1]{\bbP_n^{#1}}
\newcommand{\Wf}[1]{W_{#1}}
\newcommand{\Wfr}[1]{\hat{W}_{#1}}

\newcommand{\Znx}{Z_{n,x}}
\newcommand{\xl}{\xi_\lambda}
\newcommand{\xln}{\xi_{\lambda_0}}

\newcommand{\Kl}{{\mathbf K}_\lambda}
\newcommand{\Kln}{{\mathbf K}_{\lambda_0}} 
\newcommand{\Ul}{{\mathbf U}_\lambda}

\newcommand{\Dl}[1]{\bbD_{#1}}
\newcommand{\Hl}[1]{\bbH_{#1}}
\newcommand{\be}{\begin{equation}}
\newcommand{\ee}{\end{equation}}

\begin{document}

\title{Ballistic Phase of Self-Interacting Random Walks }
\author{Dmitry Ioffe and Yvan Velenik\\[2mm]Technion and Universit\'e de Gen\`eve}
\maketitle

\begin{abstract}
We explain a unified approach to a study of ballistic phase for 
a large family of self-interacting random walks with a drift and 
self-interacting polymers with an external stretching force. The 
approach is based on a recent version of the Ornstein-Zernike theory
developed in \cite{CaIoVe1,CaIoVe2,CaIoVe3}. It leads to  local limit
results for various observables (e.g., displacement of the end-point
or number of hits of a fixed finite pattern) on paths of
 $n$-step walks (polymers) on all 
possible deviation  scales from CLT to LD. The class of models, which 
display ballistic phase in the ``universality class'' discussed 
in the paper, includes self-avoiding walks, Domb-Joyce model, random walks in 
an annealed random potential, reinforced polymers and weakly reinforced
random walks.
\end{abstract}
\section{Introduction and results}

Self-interacting polymers and random walks have received much attention by both physicists and probabilists. As the resulting models are non-markovian, their analysis requires new techniques, and many basic questions remain open.

\citet{Chayes} recently pointed out that the presence of an arbitrary non-zero drift turns a Self-Avoiding Walk on $\bbZ^d$ into a ``massive'' model, and used this observation to prove the existence of a positive speed, using renewal techniques of Ornstein-Zernike--type, in the spirit of~\citep{ChayesChayesSAW}. A similar approach was used by \citet{Flury} in order to study non-directed polymers in a quenched random environment, under the influence of a drift: after suitable reduction to an annealed setting (at the cost of considering two interacting copies of the polymer), he used the same Ornstein-Zernike approach in order to prove equality of quenched and annealed free energies. Notice that introducing a drift in such polymer models is very reasonable from the point of view of Physics, as it is equivalent to pulling the polymer with a constant force at one endpoint, the other being pinned.

Ornstein-Zernike renewal techniques have known considerable progress in recent years, allowing nonperturbative control of percolation connectivities~\citep{CaIo}, Ising correlation functions~\citep{CaIoVe1,CaIoVe2} and more generally connectivities in random cluster models~\citep{CaIoVe3}, in the whole subcritical regime.

The purpose of this note is to point out that 
the approach of \citep{CaIoVe1,CaIoVe2,CaIoVe3}
allows a unified and  a detailed treatment of self-interacting polymers
 or self-interacting walks in ballistic regime. In particular, we shall 
 explain how to use it in order to prove that,
\begin{itemize}
\item For self-interacting polymers with a repulsive interaction (see below), a positive drift gives rise to a ballistic phase, in a very strong sense (local limit theorem for the endpoint).
\item For self-interacting polymers with an attractive interaction (see below), there is a sharp transition between a confined phase and a ballistic phase as the drift increases; in the latter regime, we establish again a local limit theorem for the endpoint.
\item These behaviours are stable under small perturbations (in a sense defined below). In addition to allowing us to consider mixed attractive/repulsive models, this also makes it possible to prove local limit theorems for some models of reinforced random walk with drift, strengthening the recent results obtained by \citet{Remco} using the lace expansion.
\item In all cases mentioned, the local limit theorem can be readily complemented by a functional CLT for the ballistic path.
\item In the ballistic phase, these results can be complemented by sharp local limit theorems for general local observables of the path. This 
permits for example a strong version of Kesten's pattern 
Theorem~\citep{Kesten} in this regime.
\end{itemize}
We would also like to remark that
the approach could  be pushed to analyse the interaction of paths as,
e.g., considered in~\citep{Flury}, in order to analyse
random polymers in quenched environment.

%It should be emphasized that our approach is very robust.
 For the sake of readability, we make 
some simplifying assumptions in the sequel (for example, on the nearest neighbour nature of paths or on
the form of the local observables), but the technique 
is flexible enough to be applicable in numerous other situations.

\subsection{Class of models}
For each nearest neighbour path $\gamma = (\gamma (0), \dots ,\gamma (n))$
on $\Zd$ define:
\begin{itemize}
 \item Length $|\gamma | = n$ and 
 displacement  $D(\gamma ) = \gamma (n) -\gamma (0)$.
\item Local times: For a given a site $x\in\Zd$ set 
\[
 l_x (\gamma )\, = \, \sum_{k=0}^{|\gamma |} \1_{\{ \gamma (k) =x\}} .
\]
Similarly, one defines local times $l_b (\gamma )$ for either
un-oriented or oriented nearest neighbour bonds $b$.
\item Potential $\Phi (\gamma )$: 
In this paper
we shall concentrate on potentials $\Phi$ which depend only on bond
or edge local times of $\gamma$.  That is, 
\be 
\label{Potential}
\Phi (\gamma )\, =\, \sum_{x\in\Zd}\phi\lb l_x (\gamma )\rb\quad\text{or}
\quad 
\Phi (\gamma )\, =\, \sum_{b}\phi\lb l_b (\gamma )\rb .
\ee
\end{itemize}
To every $h\in\Rd$ and $\lambda\in\bbR$, we associate grand-canonical weights defined by
\be
\label{hlWeights}
\Wf{h ,\lambda} (\gamma )\, =\, 
{\rm e}^{-\Phi (\gamma ) + (h ,D (\gamma )) -\lambda |\gamma |} , 
\ee
where $(\cdot ,\cdot )$ is the usual scalar product on $\Rd$. 
If $h$ or $\lambda$ equal to zero, then the corresponding entry is dropped from
the notation.

The important assumptions are those imposed on the potential $\Phi$ in \eqref{Potential}.
In all the models we are going to consider, 
\smallskip

\noindent
{\bf (N)} $\phi (0) = 0$ and $\phi$ is a non-negative and non-decreasing 
function on $\bbN$.  
\smallskip

\noindent
Furthermore, we shall distinguish between the repulsive case, 
\smallskip

\noindent
{\bf (R)} For all $l,m\in\bbN$,  $\phi (l+m ) \geq \phi (l) +\phi (m)$;
\smallskip

\noindent
and the attractive case, 
\smallskip

 \noindent
{\bf (A)} For all $l,m\in\bbN$,  $\phi (l+m ) \leq \phi (l) +\phi (m)$.
\smallskip

\noindent
In the attractive setup, there is no loss of generality in restricting our attention to sublinear
potentials,
\smallskip

 \noindent
{\bf (SL)} $\lim_{n\to\infty}\phi (n) /n =0$.
\smallskip

\noindent
Finally, we shall consider small perturbations of the above two pure 
(attractive or repulsive) cases.

\noindent
Whichever model we are going to consider, the main object of our study 
is the canonical path measure, 
\[
 \Pnf{h}\lb \cdot \rb\, =\, \frac{\Wf{h}(\cdot )}{\Znf{h}} , 
\]
where, 
\[
 \Znf{h}\, =\, \sum_{\gamma (0)=0 ,|\gamma | =n} \Wf{h} (\gamma ) .
\]
Note that in the attractive case condition {\bf (SL)} simply boils down to a choice of normalisation for the partition functions $\Znf{h}$.
\subsection{Collection of examples}
Two main examples of repulsive interactions are
\begin{enumerate}
 \item Site or bond self-avoiding walks with $\phi (l) = \infty\cdot
\delta{\{ l> 1\}} $.
\item Domb-Joyce model with $\phi (l) = \beta l(l-1)/2$ .
\end{enumerate}
Two examples of attractive interactions are
\begin{enumerate} 
\item Annealed random walks in random potential, 
\[
 \phi (l)\, =\, -\log\bbE {\rm e}^{lV},
\]
where $V$ is a {\em non-positive} random variable. 
\item Edge or site reinforced polymers with
\[
 \phi (l)\, =\, \sum_1^l \beta_k , 
\]
where $\{\beta_k\}$ is a {\em non-negative} and {\em non-increasing} sequence.
\end{enumerate}
\subsection{Connectivity constants, Lyapunov exponents and Wulff shapes}  The 
connectivity constant
\[
 \lambda_0\, \df\, \lim_{n\to\infty}\frac{\log\Zn}{n}
\]
is well defined and finite 
 in both repulsive and attractive cases. Indeed, since $\phi$ is 
non-negative and non-decreasing, 
\[
 {\rm e}^{-\phi (1)n}\, \leq \, \Zn\, \leq\, (2d)^n .
\]
On the other hand, 
$Z_{n+m}\leq Z_n Z_m$ in the case of repulsion, whereas 
$Z_{n+m}\geq Z_n Z_m$ in the attractive case.

As we shall check in the Appendix it is always the situation in the attractive case (under the normalisation
{\bf (SL)}) that,
\be 
\label{lambdan}
\lambda_0\, =\, \log (2d) .
\ee

\smallskip

\noindent
There are two slightly different approaches to  Lyapunov exponents which we are
going to employ: For every $x\in \Zd$, let $\calD_x$ be the family of nearest
neighbour paths from $0$ to $x$. The family $\calH_x$ comprises those
paths from $\calD_x$ which hit $x$ for the first time.
 Formally, 
\[
 \calH_x\, =\, \lbr \gamma\in\calD_x\, :\, l_x (\gamma )=1\rbr .
\]
Define
\be 
\label{LyapExp}
\Dl{\lambda} (x)\, =\, \sum_{\gamma\in \calD_x}\Wf{\lambda}(\gamma )\quad
\text{and}\quad 
\Hl{\lambda} (x)\, =\, \sum_{\gamma\in \calH_x}\Wf{\lambda}(\gamma ) , 
\ee
where, as before, $\Wf{\lambda}(\gamma ) = 
{\rm e}^{-\Phi (\gamma ) -\lambda |\gamma|}$. 
Notice that
\be 
\label{HD}
 \Hl{\lambda}(x)\,  \leq \, \Dl{\lambda}(x)\, \leq\, 
\sum_{n\geq \|x\|_1} \Zn {\rm e}^{-\lambda n} ,
\ee
which is converging for every $\lambda >\lambda_0$. 

\noindent
For $\lambda >\lambda_0$ define now
\be 
\label{xidef}
\xl (x)\, =\, -\lim_{M\to\infty}\frac1M\log \Hl{\lambda}(\lfloor Mx\rfloor)
\, =\, 
-\lim_{M\to\infty}\frac1M\log \Dl{\lambda}(\lfloor Mx\rfloor ).
\ee 
As we shall recall in the Appendix, 
\begin{proposition}
 \label{pro:LyapExp}
For every $\lambda >\lambda_0$ the function $\xl$ in \eqref{xidef} is well-defined.
Furthermore, $\xl$ is an equivalent norm on $\Rd$,  and
\be 
\label{xi:maxmin}
\max_{x\in\bbS^{d-1}}\xl (x)\, \leq\, c(d)
\min_{x\in\bbS^{d-1}}\xl (x) .
\ee
Finally, $\xln\df \lim_{\lambda\downarrow\lambda_0} \xl$ is 
identically zero in the repulsive case {\bf (R)}, 
whereas it is strictly positive in the attractive case {\bf (A)+(SL)}.
\end{proposition}
\noindent
For every $\lambda\geq\lambda_0$ define the Wulff shape,
\be 
\label{Wulff}
\Kl\, =\, \lbr h\in\Rd\, :\, (h,x )\leq \xl (x)\ \forall\, x\in\Rd\rbr .
\ee
The name Wulff shape is inherited from continuum mechanics, where $\Kl$
 is the equilibrium crystal shape once $\xl$ is interpreted to be a surface
tension. Alternatively, one can describe $\Kl$ in terms of polar norms, 
as was done, e.g., in \citep{FluryLD}: introducing the polar norm
\[
 \xi^*_\lambda (h)\, =\, \max_{x\neq 0}\frac{(h,x)}{\xl (x)}\, =\, 
\max_{\xl (x) =1} (h,x) .
\]
we see that $\Kl$ can be identified with the corresponding unit ball,
\[
 \Kl\, =\, \lbr h~:~ \xi_\lambda^* (h) \leq 1\rbr.
\]
This way or another, in view of Proposition~\ref{pro:LyapExp}, the limiting shape $\Kln$ has non-empty
interior in the attractive case, whereas $\Kln = \lbr 0\rbr$ in the
repulsive case.

\subsection{Main result}
Let $h\in\Rd\setminus \{0\}$. Assume that the potential 
$\Phi$ in \eqref{Potential} satisfies condition {\bf (N)}. In all the cases
under consideration the displacement per step $D (\gamma )/n$ satisfies
 under $\Pnf{h}$ a large deviation (LD) principle with a convex 
rate function $J_h$.  This, of course, does not say much. However, 
\smallskip 

\noindent
1) For repulsive potentials  {\bf (R)}, it is always the case
that the probability measures 
$\Pnf{h}$ are  asymptotically concentrated on ballistic trajectories:
For every $h\neq 0$ 
there exist $\bar{v} =\bar{v}_h\in\Rd\setminus \{0\}$, a constant $\kappa >0$ and
a small $\epsilon$ ($\epsilon < \|\bar{v}_h\|$), such that
\be 
\label{Ballistic}
\Pnf{h}\lb  \frac{D (\gamma )}{n} \not\in B_\epsilon (
\bar{v})\rb\, 
\leq\, {\rm e}^{-\kappa n} , 
\ee  
where $B_\epsilon (\bar{v}) = \lbr u : |u-\bar{v} | < \epsilon\rbr$. 
Furthermore, the end-point $\gamma (n)$ complies with the following strong
 local  limit type description: 
The rate function $J_h$ is
real analytic and strictly convex 
on  $B_\epsilon (\bar{v}_h)$ with a 
non-degenerate quadratic minimum  at $\bar{v}_h$. Moreover, there
exists  a strictly positive real analytic function $G  $  on $B_\epsilon (\bar{v}_h)$ such that
\be 
\label{BeAsympt}
\Pnf{h}\lb \frac{D(\gamma )}{n} =u\rb\, =\, 
\frac{G (u)}{\sqrt{ ( n )^d}} {\rm e}^{-n J_h (u )}\lb 1+\so\rb ,
\ee 
uniformly in $u\in B_\epsilon (\bar{v}_h)\cap \Zd /n$. In particular,
the displacement $D(\gamma )$ obeys under $\Pnf{h}$ a local CLT and local 
moderate deviations on all possible scales.
\smallskip

\noindent
2) In the (normalised) attractive
case {\bf (A) + (SL)}:
\begin{itemize}
 \item
 If $h\in {\rm int}\lb \Kln \rb$, 
then $D(\gamma )/n$ behaves sub-ballistically: Namely, the rate function $J_h$
 is bounded below in a neighbourhood of the origin as
$J_h (u) \geq \alpha \| u\|$ for some $\alpha =\alpha (h) >0$.
\item
If, on the other hand,
$h\not\in \Kln$, then both \eqref{Ballistic} and the local limit
description \eqref{BeAsympt} hold. 
\end{itemize}
\begin{remark} It remains an open question to determine what happens when $h\in\partial \Kln$.  Nevertheless, we shall show that in this case the rate function satisfies $J_h (0) =0$, and thus no ballistic behaviour in the sense of \eqref{Ballistic} is to be expected. It still remains to understand whether the sub-ballistic to
ballistic transition is of first order. One way to formulate this
question is: Assume that $h\in\partial \Kln$; does
 $\liminf_{\epsilon\to 0}\|\bar{v}_{ (1+\epsilon )h}\| >0$ hold or not?
In \citep{Mehra} it is claimed on the basis of simulations that the transition
is of first order in dimensions $d\geq 2$ whereas it is of second order in $d=1$.
\end{remark}

In the sequel, values of {\em positive} 
constants $\delta ,\epsilon, \nu, c ,C_1, c_2 ,\dots $ 
 may vary from section to section. 

\subsection{Acknowledgements}
This work was partly supported by the Swiss NSF grant \#200020-113256.

\section{Coarse Graining} Let $\lambda >\lambda_0$.  In this section we are 
going to develop a rough description of typical paths which contribute to
$\Hl{\lambda} (x)$. Their contribution is going to be measured in terms of the
probability distribution
\[
 \Plf{x}\lb \gamma \rb\, =\, \frac{\Wf{\lambda} (\gamma )}{\Hl{\lambda} (x)}
\]
on $\calH_x$. Our results will be particularly relevant in the 
asymptotic regime $\|x\|\to\infty$. 

The key idea behind the renormalisation is that although on the
microscopic scale paths $\gamma$ are entitled to wiggle as much as they wish, 
on large but yet finite scales they exhibit a much more rigid behaviour and, 
in a sense, go ballistically towards $x$. 

The skeleton scale $K$ is going to be chosen the same for all $x$.
Given such $K$ and a path $\gamma\in \calH_x$ we shall use $\hat{\gamma}_K$ for
the $K$-skeleton of $\gamma$. The construction of $\hat{\gamma}_K$ will be 
different 
for repulsive and attractive interactions, but in both cases $\hat{\gamma}_K$
will enjoy the following two properties:

\noindent
{\bf (P1)} $\gamma$ passes through {\em all} the vertices of $\hat{\gamma}_K$.

\noindent
{\bf (P2)} If $\eta$ is a portion of $\gamma$ which connects two vertices $u$ and $v$ of
$\hat{\gamma}_K$ and $\eta$ does not pass through any other vertex of $\hat{\gamma}_K$, then
\be 
\label{haircontrol}
\eta\, \subseteq \, K\Ul (u )\cup K\Ul (v) ,
\ee 
where $\Ul$ is the unit ball in the $\xl$ norm.

\noindent
Below we shall rely on the following inequality which holds for all the 
models we consider: Let $\lambda >\lambda_0$ be fixed. There exists 
 a sequence $\epsilon_K\to 0$, such that 
\be 
\label{Hlxi}
\Hl{\lambda} (u)\, \leq\, {\rm e}^{-\xl (u)\lb 1-\epsilon_K \rb} ,
\ee 
uniformly in $K$ and in $u\not\in K\Ul$. 

\noindent
Indeed, in the attractive case \eqref{Hlxi} follows from super-multiplicativity
of $\Hl{\lambda} (x)$ (with $\epsilon_K\equiv 0$).  In the repulsive case
 \eqref{Hlxi} follows from the estimate \eqref{RHbound} of Appendix. 

\subsection{Skeletons in the repulsive case} Let us fix a scale $K$. 
\begin{figure}[t]
\begin{center}
\includegraphics[height=4cm]{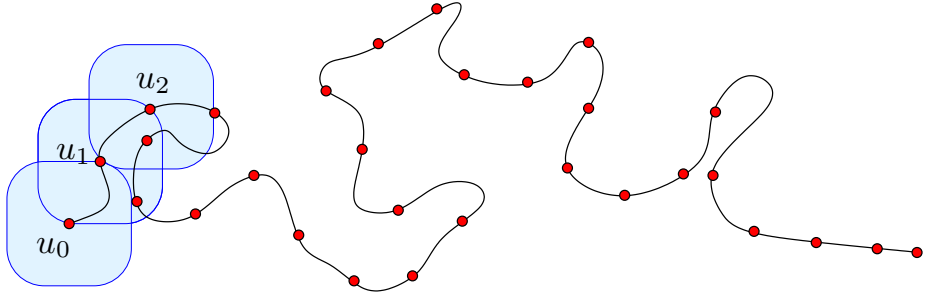}
\end{center}
\caption{The construction of the $K$-skeleton in the repulsive case.}
\label{fig_rep}
\end{figure} 
The $K$-skeleton $\hat{\gamma}_K$ 
of $\gamma = \lb\gamma (0), \dots ,\gamma (n)\rb\in\calH_x$ is constructed as follows (see Fig.~\ref{fig_rep}):
\smallskip

\noindent
\step{0}. Set $u_0=0$, $\tau_0 = 0$ and $\hat{\gamma}_K =\lbr u_0\rbr$. Go 
to \step{1}.
\smallskip

\noindent
\step{(l+1)}. If $\lb \gamma (i_l ),\dots ,\gamma (n)\rb \subseteq K\Ul (u_l )$
then stop. Otherwise set
\[
\tau_{l+1}\, =\, \min\lbr j> \tau_l~:~ \gamma (j)\not\in K\Ul (u_l )\rbr .
\]
Define $u_{l+1} = \gamma (\tau_{l+1} )$, update 
$\hat{\gamma}_K = \hat{\gamma}_K \cup \lbr u_{l+1}\rbr$ and go to 
\step{(l+2)}. \qed
\smallskip 

\noindent
Given  a skeleton $ \hat{\gamma}_K = 
(0, u_1, \dots ,u_m)$ we use $\# \lb \hat{\gamma}_K\rb =m$ for the number of {\em full} $K$-steps.
If $\gamma$ is compatible with $ \hat{\gamma}_K$, then it is decomposable as the concatenation,
\[
 \gamma \, =\, \gamma_1\cup\gamma_2\cup\dots\cup\gamma_m\cup\gamma_{m+1} ,
\]
where $\gamma_l : u_{l-1}\mapsto u_l$ for $l=1,\dots ,m$ and 
$\gamma_{m+1} :u_{m+1}\mapsto x$. 
By construction, $\gamma_l\in\calH_{u_l - u_{l-1}}$.
Since we are in the repulsive case, 
\be 
\label{Rproduct}
 \Wf{\lambda} (\gamma )\, \leq\, \prod_1^{m+1} \Wf{\lambda} (\gamma_l)
\ee
Therefore, the probability of appearance of $\hat{\gamma}_K$ is bounded
above as (see \eqref{Hlxi} for the second inequality) , 
\be 
\label{PgammaK}
\Plf{x}\lb \hat{\gamma}_K\rb\,  \leq\, 
\frac{\prod_1^m\Hl{\lambda} (u_l - u_{l-1})}{\Hl{\lambda}(x)} \, 
\leq\, \frac{{\rm e}^{-m K(1-\epsilon_K )}}{\Hl{\lambda}(x)}.
\ee
This formula is a key to a study of the geometry of typical skeletons and 
hence of typical paths. Its implications for path decomposition 
are discussed in detail in Section~\ref{Decomposition}. Notice, however, 
that one immediate consequence is the following uniform exponential bound
on the number of steps in typical skeletons: There  
exist $\nu =\nu (\lambda ,d)$
 and $C  = C (\lambda ,d )$ and a large finite  
scale  $K_0 = K_0 (\lambda ,d)$, 
such that, 
\be 
\label{StepControl}
\Plf{x}\lb \# (\hat{\gamma}_K )\, >\,  C\frac{\|x\|}{K} \rb\, \leq\, 
{\rm e}^{-\nu \|x\|} ,
\ee
uniformly in  $K\geq K_0$ and $x\in\Zd$. Indeed, the total number 
of $m$-step $K$-skeletons is bounded above as 
${\rm exp}\lbr c_1 (d,\lambda )\log K\rbr$, which, on large $K$ scales, 
is suppressed by 
 the extra  ${\rm e}^{-K (1-\epsilon_K )}$ per step price  
 in \eqref{PgammaK}.
\subsection{Skeletons in the attractive case}
Since in the attractive
case \eqref{Rproduct} does not hold we should proceed with more care. 
Accordingly our construction relies on the following fact: Assume that
the path $\gamma$ can be represented as a concatenation,
\be 
\label{gleta}
 \gamma\, =\, \gamma_1\cup\eta_1\cup\gamma_2\cup\dots\cup \gamma_m\cup\eta_m ,
\ee
such that $\gamma_1 ,\dots ,\gamma_m$ share at most one end-point. Then,
\be 
\label{Wfbound}
\Wf{\lambda}\lb\gamma \rb\, \leq\, 
\prod_{l=1}^m \Wf{\lambda} (\gamma_l ){\rm e}^{\phi (1)}
\cdot\prod_{k=1}^m {\rm e}^{-\lambda |\eta_k|} .
\ee

For every $\lambda >\lambda_0$, \eqref{lambdan} enables a comparison with a killed simple random walk.
Consequently, since $\xl$ is an equivalent norm, 
there exists $\delta = \delta (\lambda ,d)>0$, such that
\be 
\label{RWterm}
\sum_{\eta\in\calH_u}{\rm e}^{-\lambda |\eta |}\, \leq\, {\rm e}^{-\delta K} ,
\ee
uniformly in $K$ and in $u\not\in K\Ul$. 

\begin{figure}[t]
\begin{center}
\includegraphics[height=5cm]{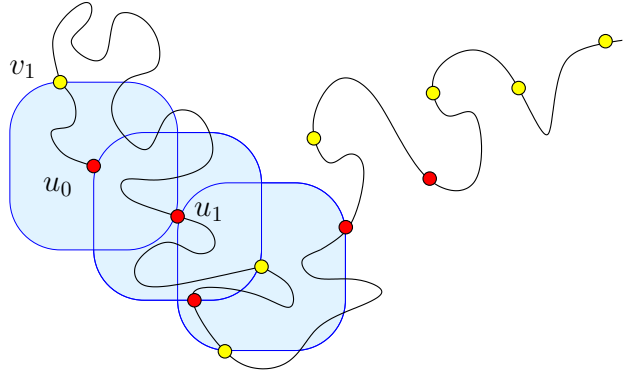}
\end{center}
\caption{The first stage of the construction of the $K$-skeleton in the attractive case (notice that $u_5=v_5$ and $u_6=v_6$ in this picture). In a second stage, one constructs skeletons for the (reversed) paths connecting $u_k$ to $v_k$, using the algorithm for the repulsive case.}
\label{fig_attr}
\end{figure} 
In view of \eqref{Wfbound} and \eqref{RWterm} it happens to be natural to 
construct skeletons $\hat{\gamma}_K$ as a union 
$\hat{\gamma}_K =\frt_K\cup\frh_K$, where $\frt_K$ is the trunk and 
$\frh_K$ is the set of hairs of $\hat{\gamma}_K $ (see Fig.~\ref{fig_attr}). Let $\gamma\in\calH_x$ and choose a scale $K$.
\smallskip

\noindent
\step{0}. Set  $u_0 =0$, $\tau_0 = 0$ and $\frt_0 =\lbr u_0\rbr$.  Go to
\step{1}. 
\smallskip

\noindent
\step{(l+1)}
If $\lb\gamma (\tau_l) ,\dots \gamma (n)\rb\subseteq K\Ul (u_l )$ then 
set $\sigma_{l+1} = n$ and stop.
Otherwise, 
 define 
\begin{equation*}
 \begin{split}
&\sigma_{l+1} = \min\lbr i >\tau_l~:~ 
\gamma (i)\not\in K\Ul (u_l )\rbr\\
&\qquad\text{ and}\\
&\tau_{l+1} = 1+ \max 
\lbr i >\tau_l~:~ 
\gamma (i) \in K\Ul (u_l )\rbr .
\end{split}
\end{equation*}
 Set $v_{l+1} = \gamma (\sigma_{l+1})$ and $u_{l+1} = \gamma (\tau_{l+1})$. 
Update $\frt_K = \frt_K\cup\lbr u_{l+1}\rbr$ and go to \step{(l+2)} \qed

\noindent
Apart from producing $\frt_K$ the above algorithm leads to a decomposition
of $\gamma$ as in \eqref{gleta} with 
\[
\gamma_l\, =\, \lb \gamma(\tau_l ),\dots ,\gamma (\sigma_{l+1} )\rb\quad
\text{and}\quad \eta_l = \lb \gamma (\sigma_l ),\dots ,\gamma (\tau_l )\rb .
\]
The hairs $\frh_K$ of $\hat{\gamma}_K$ take into account those $\eta_l$-s
which are long on $K$-th scale. Recall that $\eta_l : v_l\mapsto u_{l}$.
It is equivalent, but more convenient to think about $\eta_l$ as of 
a reversed path from $u_l$ to $v_l$. 
Then the $l$-th hair $\frh_K^l$ of $\gamma$ is constructed as follows: 
If $\eta_l\subseteq K\Ul (u_l )$ then $\frh_K^l =\emptyset$. Otherwise, 
construct $\frh_K^l$  of $\eta_l$ 
following exactly the same rules as in 
the construction of  $K$-skeletons in the repulsive case.

Putting everything together, using \eqref{Wfbound} and \eqref{RWterm}, we arrive to the following upper bound on the
probability of a $K$-skeleton $\hat{\gamma}_K = \frt_K\cup\frh_K$, 
\be 
\label{PthK}
\Plf{x}\lb \frt_K\cup\frh_K\rb\, \leq\, 
\frac{{\rm e}^{-
\# \lb \frt_K \rb K(1-\phi (1)/K)
 -\delta K \#\lb\frh_K\rb  }}{\Hl{\lambda}(x)},
\ee
where $ \#\lb\frh_K \rb\df \sum_1^m \#\lb\frh_K^l\rb$ is the total $K$-length of hairs attached to the trunk $\frt_K$.

\noindent
As in the repulsive case, \eqref{PthK} leads to an exponential upper bound on the number of $K$-steps in the trunk $\frt_K$: There exists a large finite scale  $K_0 = K_0 (\lambda ,d)$ such that
\be
\label{Atrunk}
\Plf{x}\lb \frt_K\rb\, \leq\, 
\frac{{\rm e}^{-\# (\frt_K )K \lb 1-
\epsilon_K \rb}}
{\Hl{\lambda}(x)} ,
\ee
for all $K\geq K_0$ and uniformly in $x$ and $\frt_K$. 
Consequently, there exist $\nu =\nu (\lambda ,d)$ and $C = C(\lambda ,d )$ such that,
\be 
\label{StepControlA}
\Plf{x}\lb \# (\frt_K )\, >\, C\frac{\|x\|}{K} \rb\, \leq\, 
{\rm e}^{-\nu \|x\|} ,
\ee
uniformly in $K\geq K_0$ and $x\in\Zd$. Furthermore, since the number of different ways to attach hairs $\frh_K$ with
$
 \# \lb \frh_K\rb = r
$
to vertices of a trunk $\frt_K$ of cardinality $\# \lb \frt_K\rb =m$ is bounded above as
\[
 {\rm e}^{ c_1 r\log K}\cdot\frac{1}{\max_p p^r (1-p )^m}\, \leq\, 
{\rm e}^{c_2 r\log K + m/K} ,
\]
formula \eqref{PthK} implies that
\[
 \Plf{x}\lb \# \lb \frh_K\rb =r ;
\#\lb \frt_K\rb \leq C\frac{\|x\|}{K}\rb\, \leq\, {\rm e}^{-\delta r K + 
c_3\|x\| \log K/K}
\max_{K\#(\frt_K )\leq C\|x\| }\frac{{\rm e}^{-K\# (\frt_K )}}{\Hl{\lambda}(x)} .
\]
As we shall point out in the beginning of Section~\ref{Decomposition}, 
\be 
\label{Correction}
\limsup_{K\to\infty}\limsup_{\|x\|\to\infty} \frac{1}{\|x\|}\log
\max_{K\# (\frt_K)\leq C\|x\|}
\frac{{\rm e}^{-K\# (\frt_K )}}{\Hl{\lambda}(x)} = 0.
\ee 
Consequently, for each $\epsilon > 0$ there exists a finite scale
$K_0 = K_0 (\epsilon ,\lambda ,d)$, such that
\be 
\label{HairControlA}
\Plf{x}\lb \# (\frh_K )\, >\, \epsilon\frac{\|x\|}{K} \rb\, \leq\, 
{\rm e}^{-\epsilon\delta \|x\|/2} ,
\ee
uniformly in $K\geq K_0$ and $x\in\Zd$.

\section{Irreducible decomposition of ballistic paths}
\label{Decomposition}
In this section we derive an irreducible representation of typical (under $\Plf{x}$)  paths $\gamma\in\calH_x$. This irreducible representation has an effective 1D structure and it enables a local limit treatment of various observables over paths such as, e.g., the displacement $D (\gamma )$ along $\gamma$, or the number of steps $|\gamma|$ in $\gamma$.  It should be kept in mind that in the framework of the theory we develop there are {\em many alternative}
 ways to define irreducible paths and, 
accordingly, to study statistics of other local patterns over $\gamma$.

In the sequel, given $x\in\Zd$, we say that $h\in\partial\Kl$ is dual to $x$ if $(h,x)=\xl(x)$.  Of course, if $h$ is dual to $x$, then $h$ is dual to $\alpha x$ for any $\alpha >0$. Thus it makes sense to talk about $\|x\|\to\infty$ for a fixed dual $h$.

\subsection{Surcharge function and surcharge inequality}
We shall now analyse the geometry of typical $K$-skeletons. In order to unify as much as possible our treatment of the repulsive and attractive cases, let us define the trunk $\frt_K$ of the skeleton $\hat\gamma_K$ in the repulsive case as being the skeleton itself, $\frt_K=\hat\gamma_K$.

The basic quantity in the following analysis is the \emph{surcharge function} of a trunk $\frt_K=(u_0,\ldots,u_m)$ defined by
\[
\sur(\frt_K) = \sum_{i=1}^m \sur(u_i-u_{i-1}),
\]
where the function $\sur:\bbR^d\to\bbR_+$ is given by $\sur(y) = \xl(y) - (h,y)$.
Note that $\sur (y) = 0$ if and only if $h\in\partial \Kl$ and $y$ are dual directions.
The surcharge measure $\sur$ enters our skeleton calculus in the 
following fashion: By construction, 
$\xl (u_i -u_{i-1}) \leq K +c_1 (\lambda)$
 for all $K$-steps of $\frt_K$.  As a result, 
\[
  K \# (\frt_K )\, \geq \, \sum \lb \xl (u_i -u_{i-1}) -c_1 (\lambda )\rb \, 
\geq \, \xl (x) + \sur (\frt_K ) - c_1\# (\frt_K ).
\]
Since we can restrict attention to $\# (\frt_K )\leq C\|x\|/K$, \eqref{Correction} is an immediate corollary.
 Moreover, arguing 
similarly as in~\eqref{PgammaK} and~\eqref{Atrunk}, we obtain that, for any $\epsilon >0$ fixed,
\[
\Plf{x}(\frt_K) \leq e^{-\sur(\frt_K)(1-o_K(1))},
\]
uniformly in $\|x\|$ and in  $\sur(\frt_K) \geq \epsilon \|x\|$,
with $\lim_{K\to\infty} o_K(1)=0$. The following \emph{surcharge inequality} is at the core of our method, allowing to reduce the characterisation of typical trunks to geometrical considerations.
\begin{lemma}
\label{SurchargeInequality}
For every small $\epsilon>0$ there exists $K_0(d,\lambda,\epsilon)$ such that
\[
\Plf{x} ( \sur(\frt_K) > 2\epsilon \|x\| ) \leq e^{-\epsilon \|x\|},
\]
uniformly in $x\in\bbZ^d$, $h\in\partial\Kl$ dual to $x$, and scales $K>K_0$.
\end{lemma}
\begin{proof}
By~\eqref{StepControl} and~\eqref{StepControlA}, we can assume that the trunk is admissible, that is $\# (\frt_K )\leq C\|x\|/K$. 
Since the number of such trunks is bounded by
\[
\exp \left( c_2(d,\lambda) \frac{\log K}K \|x\| \right),
\]
we infer that
\[
\Plf{x} ( \sur(\frt_K) > 2\epsilon \|x\| ) \leq e^{-\nu \|x\|} + \exp  \left( c_2(d,\lambda) \frac{\log K}K \|x\| - \sur(\frt_K)(1-o_K(1)) \right) \leq e^{-\epsilon \|x\|},
\]
as soon as $K$ is chosen large enough.
\end{proof}
Thanks to the surcharge inequality, we can exclude whole families of trunks simply by establishing a lower bound as above on their surcharge function.

\subsection{Cone points of trunks}
Let us fix $\delta\in(0,\tfrac13)$. For any $h\in\partial\Kl$, we define the  \emph{forward cone} by
\[
\fcone = \{ y\in\bbZ^d \,:\, \sur(x) < \delta\xl(x) \},
\]
and the \emph{backward cone} by $\bcone = -\fcone$.

Given a trunk $\frt_K=(u_0,\ldots,u_m)$, we say that the point $u_k$ is an \emph{$(h,\delta)$-forward cone point} if
\[
\{u_{k+1},\ldots,u_m\} \subset u_k + \fcone.
\]
Similarly, $u_k$ is  an \emph{$(h,\delta)$-backward cone point} if
\[
\{u_0,\ldots,u_{k-1}\} \subset u_k + \bcone.
\]
Finally, $u_k$ is  an \emph{$(h,\delta)$-cone point} if it is both an $(h,\delta)$-forward and an $(h,\delta)$-backward cone point.

Proceeding now as in Section~2.6 of~\citep{CaIoVe3}, we prove that most vertices of typical trunks are $(h,\delta)$-cone points. Let us denote by
$\#_{h,\delta}^{\scriptscriptstyle\rm non-cone}(\frt_K)$ the number of those vertices  in the trunk $\frt_K$ that are not $(h,\delta)$-cone points.
\begin{lemma}
Let $\delta\in(0,\tfrac13)$ be fixed. Then
\[
\sur(\frt_K) \geq c_3
\delta K\#_{h,\delta}^{\scriptscriptstyle\rm non-cone}(\frt_K),
\]
uniformly in $x\in\bbZ^d$, $h\in\partial\Kl$ dual to $x$, and $K$ large enough. In particular, the estimate
\[
\Plf{x} \lb\#_{h,\delta}^{\scriptscriptstyle\rm non-cone}(\frt_K) \geq \epsilon\,
\# (\frt_K )\rb
 \leq e^{-c_4 \epsilon \|x\|},
\]
holds uniformly in $x\in\bbZ^d$, $h\in\partial\Kl$ dual to $x$, and $K$ large enough.
\end{lemma}

\subsection{Cone points of skeletons}
In the repulsive case, the previous lemma provides all the control we need. In the attractive case, to which we restrict ourselves temporarily, it is also necessary to control the hairs.
Let us start by extending the notion of $(h,\delta)$-cone points from trunks to full skeletons: A point $u_k$ of a trunk $\frt_K=(u_0,\ldots,u_m)$ is an $(h,\delta)$-forward cone point of the skeleton $\hat\gamma_K=(\frt_K\cup\frh_K)$ if
\[
\hat\gamma_K \subset \left( u_k + \dfcone \right) \cup \left( u_k + \dbcone \right).
\]
(Notice that we increased the aperture of the cone from $\delta$ to $2\delta$.) It readily follows~\eqref{HairControlA} that most vertices of the trunk have no hair attached to them. Therefore, the only way an $(h,\delta)$-cone point of $\frt_K$ may fail to be an $(h,\delta)$-cone point of $\hat\gamma_K$ is when another vertex of the trunk has such a long hair attached to it that the latter exits the (enlarged) cone; in such a case, we say that the $(h,\delta)$-cone point is \emph{blocked}. Lemma~2.5 of~\citep{CaIoVe3} implies that most vertices of the trunk are $(h,\delta)$-cone points of $\hat\gamma_K$.
\begin{lemma}
Let $\#_{h,\delta}^{\scriptscriptstyle\rm blocked}(\hat\gamma_K)$ be the number of vertices of the trunk of $\hat\gamma_K$ that are not $(h,\delta)$-cone points of $\hat\gamma_K$.
Suppose that $\frt_K$ is an admissible trunk with $\#_{h,\delta}^{\scriptscriptstyle\rm non-cone}(\frt_K) < \epsilon\, \#(\frt_K)$. Then there exists $c_5(d,\delta,\lambda)>0$ such that
\[
\Plf{x} (\#_{h,\delta}^{\scriptscriptstyle\rm blocked}(\hat\gamma_K) \geq \epsilon\, \#(\frt_K) \,|\, \frt_K ) \leq e^{-c_5\epsilon \|x\|},
\]
uniformly in $x\in\bbZ^d$, $h\in\partial\Kl$ dual to $x$, and $K$ large enough.
\end{lemma}

\subsection{Cone points of paths}
We return now to the general case of attractive or repulsive potentials; in the latter case we identify the notions of cone points of trunks and skeletons, as these two notions coincide.

Now that the skeletons are under control, we can turn to the microscopic path itself. Let $\gamma=(\gamma(0),\ldots,\gamma(n))\in\calH_x$. We say that $\gamma(k)$ is an {$(h,\delta)$-cone point} of $\gamma$ if
\[
\gamma \subset \left( \gamma(k) + \tfcone\right) 
\cup \left( \gamma(k) + \tbcone \right).
\]
Notice that we increased again the aperture of the cone to $3\delta$. Notice 
also that we tacitly assume that  $\tfcone$ contains a lattice direction.

Proceeding as in Section~2.7 of~\citep{CaIoVe1}, we can show that, up to exponentially small $\Plf{x}$-probabilities, a uniformly strictly positive fractions of the $(h,\delta)$-cone points of the trunk of $\hat\gamma_K$ are actually $(h,\delta)$-cone points of the path.
\begin{theorem}
\label{ConePointsOfPaths}
Let $\#_{h,\delta}^{\scriptscriptstyle\rm cone}(\gamma)$ be the number of $(h,\delta)$-cone points of $\gamma$. There exist $\delta\in(0,\tfrac13)$ and two positive numbers $c$ and $\nu$, depending only on $d,\delta$ and $\lambda$, such that
\[
\Plf{x} (\#_{h,\delta}^{\scriptscriptstyle\rm cone}(\gamma) < c \|x\| ) \leq e^{-\nu \|x\|},
\]
uniformly in $x\in\bbZ^d$, $h\in\partial\Kl$ dual to $x$, and $K$ large enough.
\end{theorem}
\begin{remark}
\label{RemCone}
Observe that the above estimate still holds true (possibly for a smaller constant $\nu$) when $x$ is replaced by an arbitrary site $y\in\fcone$.
\end{remark}
\subsection{Decomposition of paths into irreducible pieces}
With the help of the Theorem~\ref{ConePointsOfPaths}, we can finally construct the desired decomposition of typical ballistic paths into irreducible pieces.

\begin{figure}[t]
\begin{center}
\includegraphics[height=4cm]{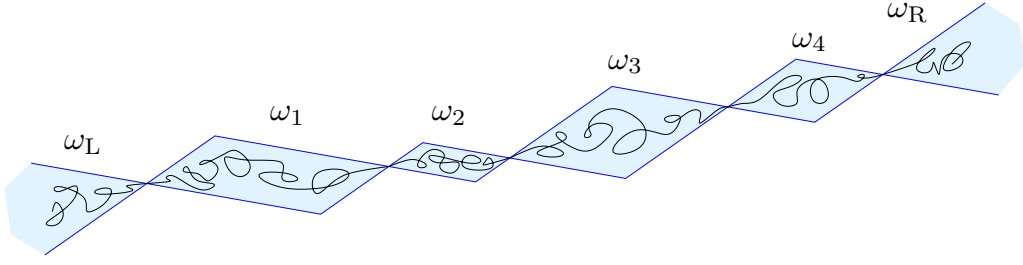}
\end{center}
\caption{The decomposition of a path into irreducible components.}
\end{figure} 
A path $\gamma=(\gamma(0),\ldots,\gamma(n))$ is said to be \emph{$(h,\delta)$-backward irreducible} if $\gamma(n)$ is the only $(h,\delta)$-cone point of $\gamma$. Similarly, $\gamma$ is said to be \emph{$(h,\delta)$-forward irreducible}  if $\gamma(0)$ is the only $(h,\delta)$-cone point of $\gamma$. Finally, $\gamma$ is said to be \emph{irreducible} if $\gamma(0)$ and $\gamma(n)$ are the only $(h,\delta)$-cone points of $\gamma$.

Let $\Omega_{\rm L} ,\Omega$ and $\Omega_{\rm R}$ be 
the corresponding sets of such
irreducible paths.

In view of Theorem~\ref{ConePointsOfPaths}, we can restrict our attention to paths possessing at least $c\|x\|$  $(h,\delta)$-cone points, at least when $\|x\|$ is sufficiently large. We can then unambiguously decompose $\lambda$ into irreducible sub-paths,
\begin{equation}
\label{DecompositionIntoIrreduciblePieces}
\gamma = \omega_{\rm L} \cup \omega_1 \cup \cdots \cup \omega_m \cup \omega_{\rm R}.
\end{equation}
We thus have the following expression
\begin{equation}
\label{DecompositionZx}
e^{(h,x)}\Hl{\lambda}(x) = O\lb e^{-\nu \|x\|}\rb \, +\, 
\sum_{m\geq c\|x\|}\sum_{\omega_{\rm L}\in\Omega_{\rm L}}
\sum_{\omega_1, \dots \omega_m\in\Omega}
\sum_{\omega_{\rm R}\in\Omega_{\rm R}}
\Wf{h ,\lambda }(\gamma )\1_{\lbr D (\gamma )=x\rbr} .
\end{equation}
In fact,  since $e^{(h,y)}\Hl{\lambda}(y)\asymp{\rm e}^{-\sur (y)}$, 
we, in view of Remark~\ref{RemCone}, 
  infer that, perhaps for a smaller choice of $\nu >0$,
\begin{equation}
\label{DecompositionZy}
e^{(h,y)}\Hl{\lambda}(y) = O \lb e^{-\nu \|y\|}\rb \, +\, 
\sum_{m\geq c\|y\|}\sum_{\omega_{\rm L}\in\Omega_{\rm L}}
\sum_{\omega_1, \dots \omega_m\in\Omega}
\sum_{\omega_{\rm R}\in\Omega_{\rm R}}
\Wf{h ,\lambda }(\gamma )\1_{\lbr D (\gamma )=y\rbr} .
\end{equation}
uniformly in $y\in \Zd$.
\subsection{Probabilistic structure of the irreducible decomposition}
\label{ssec_ProbabilisticStructureOfPaths}
We shall treat $\Omega,\, \Omega_{\rm L}$ and $\Omega_{\rm R}$ as probability spaces. In this way various path observables such as, e.g., $D (\omega )$ or $|\omega|$ will be naturally interpreted as random variables. First of all let us try to  rewrite the weights of concatenated paths in the product form,
\be 
\label{productweights}
\Wf{h ,\lambda }(\gamma )\, =\, 
\Wf{h ,\lambda }( \omega_{\rm L} \cup \omega_1 \cup \cdots \cup \omega_m \cup \omega_{\rm R})\, =\, 
\Wfr{h ,\lambda }(\omega_{\rm L})
\Wfr{h ,\lambda }(\omega_{\rm R})
\prod_l
\Wfr{h ,\lambda }(\omega_l )
\ee
If the potential $\Phi$ depends on bond local times only, then $\Wfr{h ,\lambda } = \Wf{h ,\lambda }$ qualifies. However, when $\Phi$ depends on site local times, such a choice would imply overcounting of local times at end-points. In that case, we thus set $\Wfr{h ,\lambda }(\omega) = {\rm e}^{\phi (1)}\Wf{h ,\lambda }( \omega )$ for $\omega\in \Omega_{\rm L}$ and $\omega\in\Omega$, and $\Wfr{h ,\lambda } = \Wf{h ,\lambda }$ on $\Omega_{\rm R}$. Evidently \eqref{productweights} is satisfied. Accordingly, we can rewrite \eqref{DecompositionZy} as 
\begin{equation}
\label{DecompositionZyh}
\begin{split}
{\rm e}^{(h,y)}\Hl{\lambda}(y) &= O\lb e^{-\nu\|y\|}\rb \, \\
&+\, 
\sum_{m\geq c\|y\|}\sum_{\omega_{\rm L}\in\Omega_{\rm L}}
\sum_{\omega_1, \dots \omega_m\in\Omega}
\sum_{\omega_{\rm R}\in\Omega_{\rm R}}
\Wfr{h ,\lambda }(\omega_{\rm L})
\Wfr{h ,\lambda }(\omega_{\rm R})
\prod_1^m 
\Wfr{h ,\lambda }(\omega_l)
\1_{\lbr D (\gamma )=y\rbr} .
\end{split}
\end{equation}
uniformly in $y\in \Zd$. Since the sums over $\omega_L$ and $\omega_R$ converge (by Theorem~\ref{ConePointsOfPaths}), and 
\[
\Kl = \overline{\{h\in\bbR^d \,:\, \sum_{y\in\bbZ^d} e^{(h,y)}\Hl{\lambda}(y) < \infty\}} ,
\]
 it follows that $h\in\partial\Kl$ if and
 only if $\sum_{\omega\in\Omega} \Wfr{h ,\lambda }(\omega) = 1$, and 
thus $\Wfr{h ,\lambda }$ is a probability measure on $\Omega$. Let us
use the notation $\bbQ^{h ,\lambda} = \Wfr{h ,\lambda }$ to stipulate 
this fact. Similarly we shall use the notation 
$\bbQ^{h ,\lambda}_{\rm L}$ and $\bbQ^{h ,\lambda}_{\rm R}$ for the 
values of $\Wfr{h ,\lambda }$ on respectively $\Omega_{\rm L}$ and 
$\Omega_{\rm R}$. 

In general $\bbQ^{h ,\lambda}_{\rm L}$ and $\bbQ^{h ,\lambda}_{\rm R}$ are not probability measures. However, together with $\bbQ^{h ,\lambda}$, they display exponential tails both in the displacement variable $D (\omega )$ and in the number of steps $|\omega|$ (as well as in many other path observables of interest).
Indeed, as follows from Theorem~\ref{ConePointsOfPaths} and in view of the $\tcone$-confinement properties of irreducible paths, all three measures in question already display exponential tails in the displacement variable $D(\omega )$. On the other hand, the weights $\Wfr{h ,\lambda}$ are bounded above as 
\[
 \Wfr{h ,\lambda} (\omega)\, \leq\, {\rm e}^{\|h\|\| D(\omega )\|-\lambda |\omega|
+ \phi (1)} .
\]
Consequently, there exists $c_6= c_6 (\lambda ,d )$, such that
\[
\Wfr{h ,\lambda} \bigl( \| D(\omega )\| = \ell, |\omega | > 2(\ell\|h\| +\phi (1))\bigr)\,
\leq\, c_6{\rm e}^{-\ell\| h\| (\lambda -\lambda_0 )} .
\]
Together with the already established exponential tails of $D(\omega )$, this readily implies exponential tails for the variable $|\omega |$ as well. Exactly the same line of reasoning applies to $\bbQ^{h ,\lambda}_{\rm L}$ and $\bbQ^{h ,\lambda}_{\rm R}$.

We use $\bbQ^{h ,\lambda}_m$ for the product probability measure on $\times_1^m\Omega$. Then
\eqref{DecompositionZyh} in its final form looks like
\be 
\label{DecompositionZyFinal}
{\rm e}^{(h,y)}\Hl{\lambda}(y) = O\lb e^{-\nu\|y\|}\rb  +
\sum_{m\geq c\| y\|}
\bbQ^{h ,\lambda}_{\rm L}\star\bbQ^{h ,\lambda}_{\rm R}\star\bbQ^{h ,\lambda}_m
\lb D (\omega_{\rm L}) +D (\omega_{\rm R}) +\sum_1^m
D(\omega_l ) = y\rb.
\ee
Similarly, let $F$ be some functional on paths, e.g., $F(\gamma ) = |\gamma|$ or, for a change, $F(\gamma ) = \sum_{x}\1_{\lbr l_x (\gamma )>1\rbr}$. We then have the following expression for the restricted partition
functions, 
\be 
\label{DecompositionF}
{\rm e}^{(h,y)}\Hl{\lambda}(y; F(\gamma )=f) = O\lb e^{-\nu\|y\|}\rb
 +
\sum_{m\geq c\| y\|}
\bbQ^{h ,\lambda}_{\rm L}\star\bbQ^{h ,\lambda}_{\rm R}\star\bbQ^{h ,\lambda}_m
\lb D(\gamma ) = y ;F(\gamma )=f\rb .
\ee
The point is that, for good local functionals $F$ such as in the two instances above, sharp asymptotics simply follow from local limit properties of the product measure $\bbQ^{h ,\lambda}_m$.

\bigskip
Note that a very similar analysis applies for partition functions $\Dl{\lambda}(\cdot)$. In particular, \eqref{DecompositionF} holds for $\Dl{\lambda}(y; F(\gamma )=f$ with the very same measures $\bbQ^{h ,\lambda}_{\rm L}$, $\bbQ^{h ,\lambda}$ and the appropriate modification of $\bbQ^{h ,\lambda}_{\rm R}$.

\section{Proof of the main result}
\subsection{Large deviation rate function}
Since
\[
 \Lambda (h)\, \df\, \lim_{n\to\infty} \frac1n\log\Znf{h}
\]
is well defined (either by sub- or by super-additivity) and the
distribution of the displacement per step $D (\gamma )/n$ under 
$\Pnf{h}$ is certainly exponentially tight, the random variables 
$D (\gamma )/n$ satisfy a LD principle with rate function
\[
 J_h (u )\, =\, \sup_g\lbr (g,u ) - \Lambda_h (g)\rbr ,
\]
 where $\Lambda_h (g) = \Lambda (h+g ) - \Lambda (h)$. 

\noindent
We claim that:
\be 
\label{Knot}
\Lambda \, \equiv\, \lambda_0 \quad\text{on}\ \Kln .
\ee
Consequently, if $h\in {\rm int}\lb \Kln\rb$, then $\Lambda_h\equiv 0$ in 
a neighbourhood of the origin and hence there exists $\alpha =\alpha (h)$
 such that $J_h (u)\geq \alpha \| u\|$, which is the sub-ballistic part 
of our Main Result.

\noindent
Furthermore, we claim that if $h\not\in \Kln$, then $\Lambda$ is real 
analytic and strictly convex in a neighbourhood of $h$. In addition, 
\be 
\label{StrictConv}
\bar{v}_h\, =\, \nabla\Lambda (h)\neq 0\quad\text{and}\quad
{\rm d}^2\Lambda (h)\ \text{is non-degenerate} .
\ee
In fact, for $h\not\in\Kln$, we shall see that the log-moment generating function satisfies $\Lambda (h) = \lambda (h)$ with $\lambda = \lambda (h)$ being recovered from $h\in\partial\Kl$. Since $\Kln = \cap_{\lambda >\lambda_0}\Kl$, \eqref{Knot} is an immediate consequence: Indeed, since $\Lambda$ is convex and, by Jensen inequality, $\Znf{h}\geq\Zn\asymp {\rm e}^{\lambda_0 n}$, it satisfies $\Lambda\geq \lambda_0$.
\smallskip

\noindent
For the rest of this section we shall, therefore, focus on the case 
$h\not\in\Kln$.

\subsection{Surface $\lambda =\lambda (h) =\Lambda (h)$} Our next task is 
to explain \eqref{StrictConv}.   Let $\lambda >\lambda_0$ and $h\in\partial\Kl$.
As we shall see below, there exists $\Psi (h) >0$, such that, 
\be 
\label{Apriori}
{\rm e}^{-\lambda n}\Znf{h}\, =\, \sum_{\gamma} \Wf{h ,\lambda}(\gamma )\, =\, 
\Psi (h)\lb 1+\so \rb .
\ee 
This immediately implies that $\Lambda (h) = \lambda$.
Let us for the moment accept \eqref{Apriori} as an a priori lower
bound on ${\rm e}^{-\lambda n}\Znf{h}$ (see the paragraph 
 after \eqref{iidsum} below). Then we are able to rule out
$\Pnf{h}$-negligible $x$-s  as follows:

\noindent
1) Let $0 <l<\lambda -\lambda_0$. Then,
\[
 {\rm e}^{-\lambda n}\Znf{h}\lb (h,\gamma (n))\leq ln\rb\, \leq\, 
{\rm e}^{-(\lambda -l)n}\Zn .
\]
Since $\Zn \asymp {\rm e }^{\lambda_0n }$, the latter expression is 
exponentially small and we can ignore $x$-s such that $(h,x)\leq ln$.
\smallskip

\noindent
2) Let $(h,x ) >ln$, but $x\not\in \fcone$ (see Section~\ref{Decomposition}).
Then $Z_{n,x}^{h,\lambda} \leq Z_{x}^{h,\lambda} \leq {\rm e}^{-\nu \|x\|}$,
 and, consequently, we can ignore such $x$-s as well. 
\smallskip

\noindent
It remains to consider $x\in\fcone$ with $\| x\| > l n/\|h\| \df c_1 n$.
For such $x$-s, \eqref{DecompositionF}, or rather its modification in the case of $\Dl{\lambda}$-partition functions, implies that
\begin{align}
\Znx^{h,\lambda}
&= \sum_{\substack{\gamma\in\calD_x\\|\gamma|=n}} \Wf{h,\lambda}(\gamma)\nonumber\\
&= O\lb e^{-c_1\nu  n}\rb + \sum_{m\geq c \|x\|} \bbQ^{h,\lambda}_{\rm L} \star \bbQ^{h,\lambda}_{\rm R} \star \bbQ^{h,\lambda}_m \bigl(D(\gamma)=x, |\gamma|=n \bigr)\nonumber\\
&=
C^\prime (h ) 
\sum_{m\geq c\|x\|}
\bbQ^{h,\lambda}_m \Bigl( \sum_{i=1}^m (D(\omega_i),|\omega_i|) = (x ,n)\Bigr),
\label{iidsum}
\end{align}
where in the last line, apart from 
ignoring the $O\lb e^{-c_1\nu  n}\rb$ term, we have (
already anticipating local limit behaviour under $\bbQ^{h,\lambda}_m$)
 summed out the terms involving 
 $\bbQ^{h,\lambda}_{\rm L}$ and $\bbQ^{h,\lambda}_{\rm R}$. And 
indeed, since the random vector $V (\omega ) = (D(\omega ),|\omega |)$ has exponential moments under $\bbQ^{h,\lambda}$, we can apply the classical local CLT to~\eqref{iidsum}, which, after summing up with respect to $x$, yields the lower bound in  \eqref{Apriori}, thereby justifying the conclusions 1) and 2) above.

Let now $g\in\partial {\bf K}_\mu$ with $(g ,\mu )$ being sufficiently close to $(h, \lambda)$, say, 
\[
 |\lambda -\mu | +\| g- h\|\, <\, \frac12\min\lbr \nu ,\lambda -l \rbr .
\]
Then 1) and 2) above still describe $\Pnf{g}$-negligible $x$-s. Moreover, 
the $(g,\mu )$-modification of \eqref{iidsum} still holds: For $x\in\fcone$
 with $\| x\| \geq c_1 n$, 
\[
 \Znx^{g,\mu } \lb 1 +\so \rb\, =\, 
C^\prime (g ) 
\sum_{m\geq c_1\|x\|}
\bbQ^{h,\lambda}_m \Bigl( \sum_{i=1}^m (D(\omega_i),|\omega_i|) = (x ,n)\Bigr)
{\rm e}^{(g-h, x) - 
(\mu -\lambda) n}.
\]
(The prefactor $C'(g)$ coming as before from the summation over $\omega_{\rm L}$ and $\omega_{\rm R}$.)
Since we have employed the very same sets of irreducible paths, 
the positive function $C^\prime (g )$ is real analytic.
Now, using an argument similar to the one in Subsection~\ref{ssec_ProbabilisticStructureOfPaths}, together with the fact that the lower bound in \eqref{Apriori} holds for $(g,\mu )$ as well, we deduce that
\[
 \bbQ^{h ,\lambda}\lb {\rm e}^{(g-h, D(\omega))- (\mu -\lambda)|\omega|}\rb
\, =\, 1.
\]
In other words, define
\[
 F (g ,\mu )\, =\, \log \sum_{\omega\in\Omega}{\rm e}^{(g ,D(\omega ))
-\mu |\omega |}\, \Wf{}(\omega).
\]
We have proved:
\begin{lemma}
 Let $\lambda >\lambda_0$ and $h\in\partial \Kl$. Then there exists 
$\rho > 0$ such that the graph of the function $\mu = \Lambda (g)$
over $B_\rho ( h)$ is implicitly given by
\[
 F ( g,\mu )\, =\, 0.
\]
\end{lemma}
Since the distribution of the random vector $V$
under $\bbQ^{h ,\lambda}$ is obviously non-degenerate and has finite 
exponential moments in a neighbourhood of the origin, \eqref{StrictConv}
follows. Furthermore, there exists $\epsilon > 0$ such that
\be 
\label{Nablamap}
B_\epsilon (\bar{v}_h )\, \subset\, \cup_{g\in B_\rho (h)}\nabla\Lambda (g) .
\ee

\qed

\subsection{Local limit result in the ballistic regime}
We already know that under $\Pnf{h}$ the average displacement $D (\gamma )/n$ satisfies a LD principle with strictly convex rate function $J_h$. Thereby, \eqref{Ballistic} is justified. Let us restrict our attention to $x\in B_{\epsilon n} (n\bar{v}_h)$ and go back to \eqref{iidsum}. Obviously, $(\bar{v}_h ,1)$ is parallel to $\sfz = (\bar{w} ,\bar{t})\df \bbQ^{h ,\lambda}\lb (D(\omega),|\omega|) \rb$. Set $x=\lfloor n\bar{v}_h \rfloor$. By the local CLT and Gaussian summation formula applied to \eqref{iidsum},
\be 
\label{Znxsum}
 \Znx^{h,\lambda}\, =\, \frac{C(h)}
{\sqrt{n^d}}\, \lb 1+\so\rb .
\ee
Let $y\in\B_{n\epsilon}(n\bar{v}_h )\cap\Zd$ and  $u\df y/n\in B_\epsilon (\bar{v}_h )$. By \eqref{Nablamap}, we can find $g\in B_\rho (h)$ such that $ u =\nabla\Lambda (g )$. Set $\mu = \Lambda (g)$ and $y =\lfloor nu\rfloor$.
As in \eqref{Znxsum}, 
\be 
\label{Znysum}
Z_{n ,y}^{g,\mu}\,  = \, \frac{C (g)}{\sqrt{n^d}}\lb 1+\so\rb ,
\ee
where $C$ is a positive real analytic function 
on $ B_\rho (h)$. On the other hand, 
\[
Z_{n ,y}^{h ,\lambda}\, =\, {\rm e}^{-n\lb (g-h, u) +\Lambda (h) -\Lambda (g)\rb}
 Z_{n ,y}^{g,\mu} .
\]
It remains to notice that
\[
(g-h, u) +\Lambda (h) -\Lambda (g)\, =\, J_h (u) . 
\]
Since $C$ in \eqref{Znysum} is analytic (continuous would be enough) and 
$J_h$ has quadratic minimum at $\bar{v}_h$, the partition function 
asymptotics \eqref{Apriori} follows from Gaussian summation 
formula. \eqref{BeAsympt} is proved.\qed

\section{Perturbations by small potentials and statistics of patterns}
Let $\Phi$ be either an attractive or a repulsive potential of the
type considered above, $\lambda >\lambda_0$ and $h\in\partial\Kl$. 
Let $\calU_h\subset \Rd$ be a neighbourhood of $h$.

We would like to consider perturbations of $\Phi$ of the form, 
\be 
\label{PhiTilda}
 \widetilde{\Phi} (\gamma )\, =\, 
\Phi (\gamma )\, +\, R (\gamma , h ) .
\ee
We shall assume that for each fixed $\gamma$ the function 
$g\mapsto R(\gamma , g )$ is analytic on $\calU_h$ and that 
it is appropriately negligible: For some $\epsilon $ sufficiently
small,  
\be 
\label{Rsmall}
\sup_{g\in \calU_h}\left| R(\gamma , g )\right|\, \leq\, \epsilon |\gamma |,
\ee
simultaneously for all $\gamma$.

Furthermore, we shall assume that the perturbation $R$ is in some sense 
local. This could be quantified on various levels of generality and,
in order to fix ideas, we shall restrict our attention to the 
following case: For every $g\in\calU_h$, 
\be 
\label{Rlocal}
R(\gamma_1\cup\dots\cup\gamma_m ,g) \,=\, \sum_1^m 
R(\gamma_i ,g ) ,
\ee
whenever $\gamma_1 ,\dots ,\gamma_m$ are edge disjoint.

\subsection{An example} An immediate example is a random walk with 
small edge reinforcement. Set $F\equiv 0$. For a path
$\gamma = (x_0 ,x_1, x_2, \dots )$ define the running local times 
on {\em unoriented} bonds,
\[
 l_b^t (\gamma )\, =\, \sum_{j=0}^{t-1} \1_{\lbr b = (x_i ,x_{i+1})\rbr} .
\]
Let $\beta :\bbN \mapsto \bbR_+$ be a non-decreasing 
bounded concave function with $\beta (\infty ) =\epsilon$. Set
\[
 R (\gamma ,h)\, =\, 
- \sum_{t=0 }^{|\gamma |-1}
\log \bbE ~{\rm exp}\lbr  \beta (l_{x_i ,x_i + X_h}^t (\gamma )) -
\beta (l_{(x_i ,x_{i+1})}^t (\gamma ))\rbr ,
\]
where the random variable $X_h$ is distributed as the step of a simple random
walk with drift $h$.

% \medskip
% We would like to point out that in a recent work \cite{BR} the
% authors, using different renewal techniques, were able to prove
% weaker versions of LLN and CLT for excited walks without any drift. For the moment
% it is not clear to us whether the regime they consider falls in
% the framework of the strong ballistic behaviour as discussed in
% thew present work.

\subsection{Ballistic behaviour and local limit theory}  Let
$\widetilde{Z}_n^h$ be the partition function which corresponds 
to the perturbed interaction \eqref{PhiTilda}. We claim that 
the following generalisation of \eqref{Apriori} holds:
\begin{theorem}
Let $\Phi$, $\lambda >\lambda_0$ and $h\in\partial \Kl$ be fixed. 
Then one can choose a number $\epsilon_0 = \epsilon_0 (h , \Phi ) >0$, 
a continuous function $\rho :[0,\epsilon_0]\mapsto\bbR_+$ with 
$\rho (0) = 0$ and 
a neighbourhood $\calU_h$ of $h$, so that:

 For any $\epsilon\leq\epsilon_0$ and for any 
 (analytic) perturbation $\widetilde{\Phi} = \Phi +R $ of $\Phi$
 satisfying \eqref{Rlocal} and \eqref{Rsmall} above, there
exist  an analytic function $f$ with 
\be 
\label{fbound}
\sup_{g\in\calU_h}\left| f(g)\right| \leq\, \rho (\epsilon ) ,
\ee
 and a positive analytic function
 $\widetilde{\Psi}$ on $\calU_h$, such that the following asymptotics holds
\be 
\label{AprioriR}
{\rm e}^{-\lambda (g) n}
\widetilde{Z}_n^g\, =\, \widetilde{\Psi} (g) {\rm e}^{ n f (g )}
\lb 1+\so\rb,
\ee
uniformly in $g\in\calU_h$. The function $f$ can be recovered from the following implicit relation,
\be 
\label{FunctionF}
\log\bbQ^{ g ,\lambda (g)}\lb {\rm e}^{ f(g)|\omega | - R (\omega,g)}\rb\,
=\, 0.
\ee
\end{theorem}
The proof and the implications of the above theorem 
essentially boil down to 
a rerun of the arguments employed in the pure (repulsive or attractive)
cases: The crux of the matter is that under assumption \eqref{Rsmall}
the random variable $R(\omega,h)$ on the probability space
$(\Omega ,\bbQ^{h ,\lambda})$ has exponential tails, whereas under
 assumption \eqref{Rlocal} all the relevant asymptotics are settled 
via local limit theorems for independent random variables.  Let us briefly
 reiterate the principal steps involved:

\noindent
\step{1} Assume \eqref{AprioriR} as an a priori lower bound on ${\rm e}^{-\lambda (g) n}
\widetilde{Z}_n^g$ and use it to rule out trajectories $\gamma = (\gamma (0), 
\dots ,\gamma (n))$ which fail to comply with:
\smallskip

\noindent
1) $\min\lbr (\gamma (n) ,g),\| \gamma (n)\|\rbr \geq c_1n$ and 
$\gamma_n\in\fconeh{h}$.
\smallskip

\noindent
2) $\gamma$ has at least $m\geq c_2 n$ irreducible pieces, 
\[
 \gamma = \omega_{\rm L}\cup\omega_1\cup\dots\cup\omega_m\cup \omega_{\rm R} .
\]
\step{2} By Assumption~\eqref{Rsmall}, the random variables $D(\omega )$ and 
$|\omega |$ have exponential tails under the modified measures 
$\widetilde{\bbQ}^{g ,\lambda (g)}$, 
\[
\widetilde{\bbQ}^{g ,\lambda (g)} (\omega )\, =\, 
\frac{\bbQ^{g ,\lambda (g)} (\omega){\rm e}^{-R (\omega,g)}}{
 \bbQ^{g ,\lambda (g)}\lb {\rm e}^{-R (\cdot,g)}\rb}.
\]
Consequently, both the lower bound in \eqref{AprioriR} and then 
\eqref{AprioriR} itself follow from the local limit analysis of i.i.d.
 sums $\sum_1^m V_i$, where $V (\omega ) = (D (\omega ) ,|\omega |)$. In particular, the
 limiting  log-moment generating functions
\[
 \widetilde{\Lambda}(g)\, =\, \lim_{n\to\infty}\frac{\log \widetilde{Z}_n^g}{n}
\, \df \, \lambda (g )+ f(g) ,
\]
are analytic on $\calU_h$, whereas the asymptotic speed $\tilde{v}_n = 
\nabla\widetilde{\Lambda}(h)$ has a positive projection on $h$ by 
\step{1}.  The local limit description of the distribution of the
end-point $\gamma (n)$ under the perturbed measure 
$\widetilde{\bbP}_n^h$ follows exactly as in the pure (repulsive or attractive)
 case. In particular, we arrive to the following local limit 
description of  ballistic behaviour under  $\widetilde{\bbP}_n^h$:
 There exists $\epsilon^\prime >0$, such that:
\begin{itemize}
 \item  Outside $B_{\epsilon^\prime} (\tilde{v}_h)$, 
\[
\widetilde{\bbP}_n^h\lb \frac{D(\gamma )}{n}\, \not\in\, 
B_{\epsilon^\prime} (\tilde{v}_h)\rb  \, \leq\,
{\rm e}^{-c_3 n} .
\]
\item For $nu\in B_{n\epsilon^\prime}(n\tilde{v}_h) \cap\Zd$, 
\[ 
\widetilde{\bbP}_n^h\lb D(\gamma ) = nu\rb\, =\, 
\frac{\tilde{G} (u)}{\sqrt{n^d}} {\rm e}^{ -n\tilde{J}_h (u)}
\lb 1+\so \rb ,
\]
where the rate function $\tilde{J}_h$ 
on $B_{\epsilon^\prime} (\tilde{v}_h)$
is quadratic at its 
unique minimum $\tilde{v}_h$.
\end{itemize}
\qed
\subsection{Statistics of finite patterns and local observables}
A finite pattern is a fixed nearest neighbour (and self-avoiding in case 
of SAW-s)  path, 
\[
 \eta\,=\,  \lb u_0 ,\dots ,u_p\rb .
\]
For example, we may take $\eta$ as an elementary loop, 
\be 
\label{loop}
\eta\, =\, (0, \sfe_1 , \sfe_1 +\sfe_2 ,\sfe_2 ) .
\ee
Given a trajectory $\gamma$ define $N_\eta (\gamma )$ as the 
number of times $\eta$ is re-incarnated in $\gamma$, 
\[
N_\eta (\gamma )\, =\, \sum_{l=0}^{|\gamma | - p}\1_{\lbr (\gamma (l) ,
\dots ,\gamma (l+p )) \stackrel{{\rm m}}{=}\eta\rbr} 
\]
where $\stackrel{{\rm m}}{=}$ means matching up to a shift.  Obviously, 
$N(\gamma )\leq |\gamma |$. Also, in  the case of  \eqref{loop}, 
\be 
\label{Nadditive}
N_\eta (\gamma )\, =\, N_\eta (\omega_{\rm L}) + N_\eta (\omega_1 )
+\dots + N_\eta (\omega_m ) + N_\eta (\omega_{\rm R}) ,
\ee
whenever $\gamma$ is given in its irreducible representation 
\eqref{DecompositionIntoIrreduciblePieces}. 

For  general patterns $\eta$ the relation  \eqref{Nadditive} does
not necessarily hold for the particular irreducible decomposition 
which was constructed in Section~\ref{Decomposition}. We could not 
have worried  
less: First of all given {\em any} finite pattern $\eta$ we could have 
adjusted the notion of irreducible decomposition in such a way that 
\eqref{Nadditive} would become true.  Furthermore, the  import 
of \eqref{Nadditive} is a possibility to work with independent 
random variables. In the situation we consider here  sums of finitely
 dependent variables have qualitatively the same local asymptotics
 as sums of i.i.d.-s.  So, for the sake of exposition, let us assume that
 $\eta$ satisfies \eqref{Nadditive}.

But then, given some small $\delta>0$, $R (\gamma ) = \delta N_\eta (\gamma )$ qualifies as a small
perturbation of the type considered in the preceding section. Accordingly,
we are in a position to develop a standard local limit analysis of
restricted partition functions of the type
\[
 \Znf{h}\lb N_\eta (\gamma )\, =\, \lfloor nx\rfloor\rb .
\]
\begin{theorem}
 Let $\eta$ be a finite pattern, $\lambda >\lambda_0$ and $h\in\partial\Kl$.
Then there exist $x_\eta\in (0,1)$, $\epsilon >0$, $\nu >0$
 and a rate function $J_h^\eta$ on $(x_\eta -\epsilon ,x_\eta +\epsilon )$
with quadratic minimum at $x_\eta$, 
such that, 
\[
 \bbP_n^h \lb \left| \frac{N_\eta (\gamma )}{n} - x_\eta\right| \geq\epsilon\rb
\, \leq\, {\rm e}^{-\nu n} , 
\]
and, for $x\in (x_\eta -\epsilon ,x_\eta +\epsilon )$, 
\[
 \bbP_n^h \lb  N_\eta (\gamma ) = \lfloor nx\rfloor\rb\, =\, 
\frac{G_\eta (x)}{\sqrt{n}}{\rm e}^{-n J_h^\eta (x)}\lb 1+\so\rb , 
\]
where, of course, $G_\eta$ is a positive real analytic 
function on $[x_\eta -\epsilon ,x_\eta +\epsilon ]$
\end{theorem}
Note that we have used only $| N_\eta (\gamma )|\leq |\gamma |$ and 
\eqref{Nadditive}. Therefore, the above theorem holds for a wider class
of path observables, as for example, 
\[
 N (\gamma )\, =\, \sum_{b}\1_{\lbr l_b (\gamma ) >1\rbr} .
\]

\appendix
\section{Existence of Lyapunov exponents}
\subsection{Attractive case}
In the attractive case both $\Dl{\lambda}$ and $\Hl{\lambda}$ are
super-multiplicative, 
\[
 \Dl{\lambda}(x+y)\geq \Dl{\lambda}(x)\Dl{\lambda}(y)\quad\text{and}
\quad 
\Hl{\lambda}(x+y)\geq
\Hl{\lambda}(x)\Hl{\lambda}(y) .
\]
Hence both limits in \eqref{xidef} exist.

Furthermore, since the potential $\Phi$ is monotone in $\gamma$, 
$0\leq \Phi (\gamma ) \leq \Phi (\gamma\cup\eta )$, 
\be 
\label{DHbound}
 \Dl{\lambda}(x)\, \leq\, \Hl{\lambda}(x)\cdot\sum_{\eta :0\mapsto 0}
{\rm e}^{-\lambda |\eta |} .
\ee 
The relation \eqref{lambdan}  immediately implies that the right-most term in 
\eqref{DHbound} is convergent for every $\lambda >\lambda_0$, and, 
consequently, that both limits in \eqref{xidef} are equal. 

\noindent
In order to check \eqref{lambdan} note first of all that since $\Phi$ is 
non-negative, the partition function $\Zn$ is bounded above
by the total number of all $n$-step nearest neighbour trajectories; 
$\Zn\leq (2d)^n$. For the reverse direction, proceeding as in~\citep{FluryLD}, let us fix a number $R$ and note that
in view of the monotonicity of $\phi$, 
\[
\Zn\, \geq\, \frac{\Zn(\gamma\subset B_R)}{(2d)^n}
{\rm e}^{n\log (2d) - |B_R\cap\Zd|\phi (n)} \, =\, 
\bbP_{{\rm SRW}}\lb \gamma\subset B_R\rb 
{\rm e}^{n\log (2d) - |B_R\cap\Zd|\phi (n)} ,
\]
where $\bbP_{{\rm SRW}}$ is the distribution of simple random walk on 
$\Zd$. 
Since~\citep{HrVe}
\[
\bbP_{{\rm SRW}}\lb \gamma\subset B_R\rb \, \geq\, {\rm exp}\lbr
-c (d)\frac{n}{R^2}\rbr ,
\]
we conclude that for all $R$ and $n$,
\[
\frac1n\log\Zn\, \geq\, \log (2d) - \frac{c(d)}{R^2} - |B_R\cap\Zd|
\frac{\phi (n)}{n} .
\]
Thus,  \eqref{lambdan} indeed follows from {\bf (SL)}.

Inequality \eqref{xi:maxmin} is a trivial consequence of convexity 
and lattice symmetries. It remains to check that in the attractive 
case $\xln$ is strictly positive. However, again as in \citep{FluryLD}, 
\[
 \Hl{\lambda}(x)\, =\, \sum_{n= \|x\|_1}^\infty \bbE_{{\rm SRW}}
{\rm e}^{-\Phi (\gamma ) - (\lambda -\lambda_0)n}\1_{\tau_x =n} ,
\]
where $\tau_x$ is the first hitting time of $x$. As a result, 
$\xl  (x) \geq \phi (1) \|x\|_1$ for all 
$\lambda\geq \lambda_0$.\qed
\subsection{Repulsive case} First of all, 
\[
 \Dl{\lambda}(x)\leq \Hl{\lambda}(x)\Dl{\lambda}(0) .
\]
Thus, in view of \eqref{HD}, for 
every $\lambda >\lambda_0$  both exponents in \eqref{xidef} are 
automatically equal once it is proved that at least one of them 
is defined. We shall show that the first limit in 
\eqref{xidef} exists, using a familiar bubble diagram method: Consider
\[
 \Hl{\lambda}(x)\Hl{\lambda}(y)\, =\, 
\sum_{\gamma\in\calH_x}\sum_{\eta\in x+\calH_y} 
{\rm e}^{-\Phi (\gamma ) -\Phi (\eta )} .
\]
Given a couple of trajectories $\gamma\in\calH_x$ and 
$\eta\in x+\calH_y$ define, 
\[
 \underline{n} = \min\lbr l~:~\gamma(l)\in \eta\rbr\quad\text{and}
\quad 
\overline{n} = \max\lbr k~:~ \eta (k) = \gamma (\underline{n} )\rbr .
\]
The concatenated path
\[
 \tilde{\gamma }\, \df\, \lb\gamma (0), \dots ,\gamma (\underline{n} ), 
\eta(\overline{n} +1), \dots ,\eta (|\eta|)\rb\, \df\, 
\gamma_1\cup\eta_1 \, \in\, \calH_{x+y} .
\]
Define also 
\[
 \gamma_2 = \lb \gamma (\underline{n}), \dots, \gamma (|\gamma |)\rb
\quad\text{and}\quad 
\eta_2 = \lb \eta (0), \dots ,\eta (\overline{n})\rb .
\]
Evidently, one can recover $\gamma$ and $\eta$ from $\tilde{\gamma}$, 
$\gamma_2$ and $\eta_2$ up to, perhaps, an interchange of $\gamma_2$ and
$\eta_2$.  Now, since $\gamma_1$ and $\eta_1$ intersect only at the 
end-point, 
\begin{equation*}
 \begin{split}
\Phi (\gamma ) +\Phi (\eta )\, &\geq\,  \Phi(\gamma_1 )+\Phi (\gamma_2) +
\Phi(\eta_1 )+\Phi (\eta_2)\\
&=\, \Phi (\tilde{\gamma} ) - \phi (1) +\Phi (\gamma_2) +\Phi (\eta_2) .
\end{split}
\end{equation*}
As a result, 
\be 
\label{RHbound}
 \Hl{\lambda}(x)\Hl{\lambda}(y)\,  \leq \, 2{\rm e}^{\phi (1)} \lb
\sum_z \Dl{\lambda}(z)^2\rb \Hl{\lambda} (x+y ) .
\ee
But $\sum_z \Dl{\lambda}(z)^2<\infty$ for every $\lambda >\lambda_0$,
and the first limit in \eqref{xidef} is indeed well defined.

It remains to show that in the repulsive case $\xln\equiv 0$.  If 
this is not the case, then, as can be easily deduced from
\eqref{RHbound}, there exists some finite constant $c$ such that
\[
 \Dl{\lambda} (x) \leq c{\rm e}^{- \xl (x)} \leq c{\rm e}^{- \xln (x)} ,
\]
uniformly in $\lambda >\lambda_0$ and $x\in\Zd$. 
 By Fatou this would mean that $\sum_x \Dl{\lambda_0} (x)$ converges.
However, since we are in the repulsive case, $Z_n\geq {\rm e}^{\lambda_0 n}$, 
and consequently,  
\[
 \sum_x \Dl{\lambda_0} (x)\, =\, \sum_n Z_n {\rm e}^{ -\lambda_0 n}\, =\, \infty.
\]
\qed

\bibliographystyle{chicago}
\bibliography{IV-revised}
\end{document}